\documentclass[12pt,a4paper]{article}
\usepackage{amsfonts, amsthm, amsmath, amssymb}

\usepackage[english]{babel}

\theoremstyle{plain}
\newtheorem*{teor}{Theorem}
\newtheorem{corollary}{Corollary}

\theoremstyle{remark}

\newcommand{\be}{\begin{equation}}
\newcommand{\ee}{\end{equation}}

\theoremstyle{plain}

\theoremstyle{plain}
\newtheorem{lemma}{Lemma}

\newcommand{\beps}{B_{\varepsilon}}

\title{Entropy, Lyapunov exponents and the volume growth of boundary
distortion under the action of dynamical systems}

\author{B.M. Gurevich \footnote{Department of
Mechanics and Mathematics, Moscow State University, and the
Institute for Information Transmission Problems, Russian Academy of
Sciences.} \ and S.A. Komech\footnote{Institute for Information
Transmission Problems, Russian Academy of Sciences.}\thanks{The work
of both authors is partially supported by RFBR grant 11-01-00485}.}

\date{}

\begin{document}
\maketitle



\section{Introduction.}
\label{introd}


Apart from the well-known studies linking the entropy of a measure
preserving smooth dynamical system with Lyapunov exponents (see
\cite{P}, \cite{K1}, \cite{K2}, \cite{LY1}, \cite{LY2}),
there is a few works dealing with geometric meaning of the
measure-theoretic entropy. One of them is \cite{BK}, where the
so-called local entropy was introduced, which turned out to coincide
with the entropy.

Another approach was used in \cite{G} for a class of symbolic
dynamical systems, but it is applicable to a much more general
situation and is as follows.

Let $f$ be a homeomorphism of a metric space $(X,\rho)$ and $\mu$ an
$f$-invariant Borel probability measure on $X$. For a point $x\in X$
we consider the $\varepsilon$-ball $B(x,\varepsilon)$ around $x$ and
treat the quantity
\begin{equation}
\label{ratio} \frac{1}{k}\ln\frac{\mu(O_{\varepsilon}(f^k
B(x,\varepsilon)))} {\mu(B(x,\varepsilon))},
\end{equation}
where $O_{\varepsilon}(A)$ is the $\varepsilon$-neighborhood of a
set $A\subset X$, as a logarithmic deformation degree of the
boundary of $B(x,\varepsilon)$ under the action of $f^k$.



It is natural to pass to the limit as $\varepsilon\to 0$ and
$k\to\infty$, but one easily sees that a nontrivial asymptotics is
possible only if there is a relation between $k$ and $\varepsilon$.
All the results were obtained when $k=k(\varepsilon)$ and
\begin{equation}
\label{conditions} \lim_{\varepsilon\to
0}k(\varepsilon)/\ln\varepsilon=0,\ \ \ \lim_{\varepsilon\to
0}k(\varepsilon)=\infty.
\end{equation}

For a subshift of finite type $(X,f)$ and for an arbitrary invariant
ergodic probability measure $\mu$ it was established in \cite{G}
that the expression (\ref{ratio}) converges in $L^1_\mu$ to
$h_{\mu}(f)$, the measure theoretic entropy of the shift
transformation $f$ with respect to $\mu$. This was generalized to
synchronized systems and hence to all sofic systems in \cite{Ko}.

For the smooth dynamical systems, precisely $n$-dimensional torus
automorphism, preserving the Lebesgue measure, the convergence of
(\ref{ratio}) to $h_{\mu}(f)$ at each point of the torus was proved
in \cite{ESAIM}.

Note that there is a special feature in the smooth case: the
existence of a natural measure on $X$, namely, the Lebesgue measure,
and it is reasonable to study the asymptotic behavior of
(\ref{ratio}) for this measure albeit it could be not invariant. We
do this for Anosov diffeomorphisms, but the result can be extended
to a wider class of smooth dynamical systems. Our main goal here is
to prove the following
\begin{teor}
Let $f$ be a $\in C^{1+\alpha}(M)$ Anosov diffeomorphisms of a
compact Riemannian manifold $M$ without a boundary and $\nu$ be an
$f$-invariant ergodic Borel probability measure on $M$. Then for any
function $k:\mathbb R^+\to\mathbb Z^+$ satisfying (\ref{conditions})
and for $\nu$-a.e. $x$,
\begin{equation}
\label{main} \lim_{\varepsilon\to 0}
\frac{1}{k(\varepsilon)} \ln\frac{\mu\left(
O^{\varepsilon}(f^{k(\varepsilon)}
B(x,\varepsilon))\right)}{\mu(B(x,\varepsilon))}=
\sum_{i:\lambda_i>0}\lambda_i d_i=:\lambda_\nu^+,
\end{equation}
where $\mu$ is the Riemannian volume, $\lambda_i$ are the Lyapunov
exponents of $\nu$ and $d_i$ are their multiplicities.
\end{teor}

\section{Proof of the Theorem.}
\label{proof}

For some $\delta>0$ introduce local stable and unstable manifolds
$W_{\delta}^s(x)$ and $W_{\delta}^u(x)$, respectively, and assume
that $\delta$ and $\varepsilon$ are so small that for all $x\in M$
the intersection $W_{\delta}^s(y)\cap W_{\delta}^u(x)$ is exactly
one point when $\rho(x,y)\leq\varepsilon$ (here $\rho$ is the
Riemannian metric). Consider the set
$$
P_u(x,\varepsilon):=\bigcup_{y\in B(x,\varepsilon)}
(W_{\delta}^s(y)\cap W_{\delta}^u(x)).
$$
We see that $P_u(x,\varepsilon)$ is the "projection" of
$B(x,\varepsilon)$ on $W_{\delta}^u(x)$ along the stable leaves
$W_{\delta}^s$.

For every $x\in M$ and $r\ge0$ denote the ball of radius $r$ on
$W_{\delta}^u(x)$ around $x$ (in the induced metric $\rho^u$) by
$B^u(x,r)$.
\begin{lemma}
\label{between} There exist $C_1,C_2 >0$ such that for all
sufficiently small $\varepsilon>0$,
$$
B^u (x,C_1\varepsilon)\subset P_u(x,\varepsilon)\subset  B^u
(x,C_2\varepsilon).
$$
\end{lemma}

The proof relies on the following two facts: 1) the angle between
$W_{\delta}^s(x)$ and $W_{\delta}^u(x)$ as a function of $x$ is
bounded away from zero; 2) there are $c>0$ and $\varepsilon>0$ such
that if $y\in W_{\delta}^u(x)$ and $\rho(x,y)\le\varepsilon$, then
$\rho^u(x,y)\le c\varepsilon$.

The following lemma will allow us to obtain upper and lower
estimates for the volume growth in an unstable manifold.
\begin{lemma}
\label{limit1} Let $\{\beps,\varepsilon>0\}$ be a family of subsets
of $W_\delta^u(x)$ such that ${\rm
diam}(\beps)\leq\gamma\varepsilon,\ \gamma>0$ and $x\in\beps$ for
each $\varepsilon$. If $f$, $k(\varepsilon)$ and $\nu$ are as in the
above Theorem, then for $\nu$-a.e. $x\in M$
\begin{equation}
\label{limit} \lim_{\varepsilon\to 0}\frac{1}{k(\varepsilon)}\ln
\frac{\mu^u(f^{k(\varepsilon)} \beps)}{\mu^u(\beps)}=\lambda^+_\nu,
\end{equation}
where $\mu^u$ is the Riemannian volume in the corresponding unstable
manifold.
\end{lemma}
\begin{proof}
For brevity we will write $k$ instead of $k(\varepsilon)$. Since the
unstable Jacobian $J^u$ is a continuous function on $M$ (see
\cite{KH}, Section 19.1), we can apply the Mean Value Theorem to
obtain a sequence of points $x_i\in f^i B_\varepsilon$, $0\le i\le
k-1$, such that
\begin{align}
\mu^u(f^k\beps)=&\int_{f^{k-1}\beps}J^u(y)\mu^u(dy) \notag\\
=&J^u(x_{k-1})
\mu^u(f^{k-1}\beps)=\dots=\mu^u(\beps)\Pi_{i=0}^{k-1}J^u(x_i).
\end{align}

By the compactness of $M$ there exists a $\beta>0$ such that
$\rho(f(y),f(z))\leq\beta\rho(y,z)$. Therefore $\rho(x_i,f^i x)\leq
\gamma\varepsilon\beta^i$. Using the fact that $J^u $ is H\"{o}lder
continuous with some exponent $\alpha>0$ and a factor $C>0$ (see
\cite{KH}, Section 19.1), we obtain
\begin{align} \label{upper}\frac{1}{k}\ln\frac{\mu^u(f^{k}
\beps)}{\mu^u(\beps)}=&\frac{1}{k}\ln\prod_{i=0}^{k-1}
J^u(x_i)\leq\frac{1}{k}\sum_{i=0}^{k-1}\ln(J^u(f^ix)+C(\varepsilon
\gamma\beta^i)^{\alpha}) \notag \\
\leq \frac{1}{k}\sum_{i=0}^{k-1}\ln & J^u(f^i x)+
\frac{1}{k}\sum_{i=0}^{k-1}\ln\left(1+\frac{C(\varepsilon
\gamma\beta^i)^{\alpha}}{J^u(f^ix)}\right)\leq\frac{1}{k}
\sum_{i=0}^{k-1}\ln J^u(f^i x)\notag\\
+\frac{1}{k}\sum_{i=0}^{k-1}\frac{C(\varepsilon
\gamma\beta^i)^{\alpha}}{J^u(f^ix)} &
\leq\frac{1}{k}\sum_{i=0}^{k-1}\ln J^u(f^i x)+\frac{C
(\varepsilon\gamma)^\alpha(\beta^{k\alpha}-1)}{k\min_{y\in
M}J^u(y)(\beta-1)}.
\end{align}
Conditions (\ref{conditions}) imply that the second term in
(\ref{upper}) tends to zero as $\varepsilon\to 0$. But the first
term tends to $\lambda^+_\nu(x)$ (for more details on the unstable
Jacobian see \cite{K1}, \cite{KH}).

Thus we have come to an upper estimates. A lower one can be obtained
similarly.
\end{proof}

\smallskip

The next lemma reflects a uniformity of the Anosov systems.
\begin{lemma}
\label{uniform} For every $a>0$ there exist $b(a)>0$ and
$\varepsilon_a>0$ such that for all $x,y\in M$ and
$\varepsilon<\varepsilon_a$,
$$
\frac{1}{b(a)}\le\frac{\mu(B(x,a\varepsilon))}{\mu(B(y,\varepsilon))}\le b(a),\
\frac{1}{b(a)}\le\frac{\mu^u(B^u(x,a\varepsilon))}
{\mu^u(B^u(y,\varepsilon))}\le b(a).
$$
\end{lemma}



We now turn immediately to the proof of the Theorem. First we
construct by induction a finite sequence of points $y_i\in
f^{k(\varepsilon)}P_u(x,\varepsilon)$, $1\le i\le N(\varepsilon)$, such that the
balls $B(y_i,\varepsilon)$ cover the set
$f^{k(\varepsilon)}P_u(x,\varepsilon)$ and
\begin{equation}
\label{eps3} B(y_i,\varepsilon/3)\cap
B(y_j,\varepsilon/3)=\varnothing,\ \ 1\le i<j\le N(\varepsilon).
\end{equation}
We start with an arbitrary $y_1$. If $y_1,\ldots,y_m$ are already
chosen and the balls $B(y_i,\varepsilon)$, $1\le i\le m$, do not
cover $f^{k(\varepsilon)}P_u(x,\varepsilon)$, take an arbitrary
non-covered point as $y_{m+1}$. Clearly,
\begin{equation}
\label{far} \rho(y_i,y_j)\ge\varepsilon\text{ when }1\le i,j\le
m+1,\ i\ne j.
\end{equation}
Hence this process will stop after a finite number of steps since
$M$ is compact. Using (\ref{far}), we also obtain (\ref{eps3}).
\smallskip

\noindent {\bf An upper bound estimate.} Lemma 1 and the fact that
$f$ is expanding along unstable manifolds imply that
\begin{equation}
\label{Cup}
O^u_{\varepsilon}(f^{k(\varepsilon)}P_u(x,\varepsilon))\subset
O_{\varepsilon}^u(f^{k(\varepsilon)}B^u(x,C_2\varepsilon))\subset
f^{k(\varepsilon)}B^u(x,(C_2+1)\varepsilon).
\end{equation}
From (\ref{eps3}), (\ref{Cup}) we obtain
\begin{equation}\label{nums}
N(\varepsilon)\leq\frac{\mu^u(f^{k(\varepsilon)}B^u(x,
(C_2+1)\varepsilon))}{\min\limits_{1\le j\le N(\varepsilon)}
\mu^u(B^u(y_j,\varepsilon/3))}.
\end{equation}

Since $f$ is contracting along stable manifolds, the sets
$f^{k(\varepsilon)}B(x,\varepsilon)$ and
$f^{k(\varepsilon)}P_u(x,\varepsilon)$ approach each other as
$\varepsilon\to 0$. Therefore for $\varepsilon$ small enough
\begin{equation}
\label{up}
O^{\varepsilon}(f^{k(\varepsilon)}B(x,\varepsilon))\subset\bigcup_{
1\le i\le N(\varepsilon)}B(y_i,2\varepsilon).
\end{equation}
By (\ref{nums}) -- (\ref{up})
\begin{align*} \frac{1}{k(\varepsilon)}
\ln\frac{\mu\left(O^{\varepsilon}(f^{k(\varepsilon)}
B(x,\varepsilon))\right)}{\mu(B(x,\varepsilon))}\leq
\frac{1}{k(\varepsilon)}\ln\frac{N(\varepsilon)\max_{1\le i\le
N(\varepsilon)}\mu(B(y_i,2\varepsilon))}
{\mu(B(x,\varepsilon))}\notag \\
\leq\frac{1}{k(\varepsilon)}\ln\frac{\mu^u(f^k B^u(x,
(C_2+1)\varepsilon))}{\mu^u(B^u(x, (C_2+1)\varepsilon))}
+\frac{1}{k(\varepsilon)}\ln
\frac{\mu^u(B^u(x,(C_2+1)\varepsilon))}{\min_{1\le i\le
N(\varepsilon)}\mu^u(B^u(y_i,\varepsilon/3))}\notag \\
+\frac{1}{k(\varepsilon)}\ln\frac{\max_{1\le i\le N(\varepsilon)}\mu
(B(y_i,2\varepsilon))}{\mu(B(x,\varepsilon))}.
\end{align*}
By Lemma \ref{limit1} the first term gives us the sum of the
positive Lyapunov exponents, while by Lemma \ref{uniform} the last
two terms tend to zero.
\smallskip

\noindent {\bf A lower bound estimate.} As is easy to verify, there
exists a constant $\overline{C}$ such that for $\delta,\varepsilon$
small enough
\begin{equation*}
\label{shar-projection} B^u (y,\varepsilon)\subset B(y,
\varepsilon)\cap W^u_\delta(y) \subset
B^u(y,\overline{C}\varepsilon), \quad y\in M.
\end{equation*}
From this we obtain
\begin{equation}
\label{center-inf} N(\varepsilon)\geq\frac{\mu^u( f^k P_u(x,
\varepsilon))}{\max_{1\le i\le N(\varepsilon)}\mu^u(B^u(y_j,
\overline{C}\varepsilon))}.
\end{equation}
Now (\ref{eps3}) and (\ref{center-inf}) yield
\begin{align*}
\frac{1}{k(\varepsilon)} &
\ln\frac{\mu\left(O_{\varepsilon}(f^{k(\varepsilon)}
B(x,\varepsilon))\right)}{\mu(B(x,\varepsilon))}
\geq\frac{1}{k(\varepsilon)}\ln\frac{N(\varepsilon)\min_{1\le i\le
N(\varepsilon)}\mu(B(y_,\varepsilon/3))}
{\mu(B(x,\varepsilon))}\ge \notag \\
\frac{1}{k(\varepsilon)} & \ln\frac{\mu^u(f^k B^u(x,
C_1\varepsilon))}{\mu^u(B^u(x, C_1\varepsilon))}+\frac{1}{k(\varepsilon)} \ln\frac{\mu^u(B^u(x,
C_1\varepsilon))}{\max_{1\le i\le N(\varepsilon)}
\mu^u(B^u(y_j,\overline{C}\varepsilon))}\notag \\
+ & \frac{1}{k(\varepsilon)}\ln\frac{\min_{1\le i\le N(\varepsilon)}\mu(B(y_,
\varepsilon/3))}{\mu(B(x,\varepsilon))}.
\end{align*}
As above, the first term tends to the sum of the positive Lyapunov
exponents, and the last two ones vanish as $\varepsilon\to 0$.

This completes the proof of the Theorem.

\begin{corollary}
If the assumptions of the above theorem are satisfied and if $\nu$
is an SBR measure for $f$ (see \cite{LY1}), then the left-hand side
of (\ref{main}) is $h_\nu(f)$.
\end{corollary}


The authors are deeply indebted to D. Burago for useful comments.

\end{document}